\documentclass[11pt]{article}
\usepackage[english]{babel} 
\usepackage[utf8]{inputenc} 
\usepackage[T1]{fontenc} 
\usepackage{lmodern} 
\usepackage{amsmath}
\usepackage{amssymb}
\usepackage{amsthm}
\usepackage{color}
\allowdisplaybreaks[1] 
\usepackage{marginnote}
\usepackage[a4paper, left=2cm, right=4cm, top=3cm, bottom=3cm]{geometry}
\usepackage{titlesec}
\usepackage{hyperref}

\titleformat{\section}[block]{\Large\bfseries}{\thesection.}{3pt}{} 
\titleformat{\subsection}[runin]{\normalfont\bfseries}{\thesubsection.}{3pt}{}[.] 

\def\RR{\mathbb{R}} 
\def\NN{\mathbb{N}} 
\def\EE{\mathbb{E}} 
\def\MM{\mathcal{M}}
\def\P{\mathcal{P}} 
\def\Q{\mathcal{Q}} 
\def\R{\mathcal{R}} 
\def\v{\mathbf{v}}
\def\V{\mathbf{V}}
\def\Z{\mathbf{Z}}
\def\z{\mathbf{z}}
\def\ind{\mathbf{1}} 
\def\ii{\mathbf{i}} 
\def\Law{\textnormal{Law}} 
\def\-{\text{-}} 

\newtheorem{thm}{Theorem}
\newtheorem{lem}[thm]{Lemma}

\newtheorem{defi}[thm]{Definition}
\theoremstyle{definition}
\newtheorem{rmk}[thm]{Remark}

\title{Uniform propagation of chaos for the thermostated Kac model}

\author{Roberto Cortez\footnote{Universidad Andr\'es Bello, Departamento de Matem\'aticas. E-mail: \texttt{roberto.cortez.m@unab.cl}. Supported by Iniciaci\'on Fondecyt Grant 11181082 and by Programa Iniciativa Científica Milenio through Nucleus Millenium Stochastic Models of Complex and Disordered Systems}
\, and Hagop Tossounian\footnote{Universidad de Chile, DIM-CMM, E-mail: \texttt{htossounian@cmm.uchile.cl}. Supported by Programa Iniciativa Científica Milenio through Nucleus Millenium Stochastic Models of Complex and Disordered Systems.
}}

\begin{document}

\maketitle

\begin{abstract}

We consider Kac's 1D $N$-particle system coupled to an ideal thermostat at temperature $T$, introduced by Bonetto, Loss, and Vaidyanathan in 2014.  We obtain a propagation of chaos result for this system, with explicit and uniform-in-time rates of order $N^{-1/3}$ in the $2$-Wasserstein metric. We also show well-posedness and equilibration for the limit kinetic equation in the space of probability measures. The proofs use a coupling argument previously introduced by Cortez and Fontbona in 2016.

\end{abstract}

\section{Introduction and main result}


\subsection{Thermostated Kac particle system}

We are interested in Kac's 1D particle system, subjected to interactions against particles taken from an ideal external thermostat, as studied for instance in \cite{bonetto-loss-vaidyanathan2014,tossounian-vaidyanathan2015}. It can be described as follows: consider $N$ particles characterized by their one-dimensional velocities, subjected to two types of random interactions:
\begin{description}
	\item[Kac collisions:] at rate $\lambda N$, randomly select two particles in the system and update their velocities $v,v_*\in\RR$ according to the rule
	\begin{equation}
	\label{eq:kac}
	\left( \begin{array}{c}
	v \\
	v_*
	\end{array} \right)
	\mapsto
	\left( \begin{array}{c}
	v' \\
	v_*'
	\end{array} \right) :=
	\left( \begin{array}{c}
	v\cos\theta - v_*\sin\theta \\
	v\sin\theta + v_*\cos\theta
	\end{array} \right),
	\end{equation}
	where $\theta \in [0,2\pi)$ is selected uniformly at random. This rule preserves the energy: $v^2 + v_*^2 = (v')^2 + (v_*')^2$.
	
	\item[Thermostat interactions:] at rate $\mu N$, randomly select a particle in the system and update its velocity $v\in\RR$ according to the rule
	\begin{equation}
	\label{eq:thermostat}
	    v \mapsto v\cos\theta - w\sin\theta,
	\end{equation}
	where $\theta$ is again selected uniformly at random on $[0,2\pi)$, and $w$ is sampled with the Gaussian density $\gamma(w) = (2\pi T)^{-1/2} e^{-w^2/2T}$. This can be seen as an interaction against a particle taken from an \emph{ideal thermostat}, that is, from an infinite reservoir at thermal equilibrium with temperature $T>0$.
\end{description}

Here $\lambda>0, \mu>0$ are given fixed constants representing the rate of Kac and thermostat collisions, respectively. The initial velocities of the $N$ particles are chosen according to some prescribed symmetric distribution $f_0^N$ on $\RR^N$, and all previous random choices are made independently. These rules unambiguously specify the law of the particle system as an $\RR^N$-valued pure-jump continuous-time Markov process, whose state at time $t\geq 0$ is denoted $\V_t = (V_t^1,\ldots,V_t^N)$, and we also write $f_t^N = \Law(\V_t)$ for its symmetric distribution. For simplicity, in our notation we omit the dependence on $N$ in the particle system $\V_t$.

Kac's original model \cite{kac1956}, corresponding to the case $\mu=0$, represents the evolution of a large number of indistinguishable particles that exchange energies via random collisions in a one dimensional caricature of a gas, as a simplification of the more realistic spatially homogeneous Boltzmann equation. The form of the collision rule \eqref{eq:kac} implies that the average energy $\frac{1}{N}\sum_i (V_t^i)^2$ is preserved a.s., and one typically assumes that the initial average energy is a.s.\ equal to $1$. Thus, $f_t^N$ is supported on the sphere $S^N = \{\v\in\RR^N : \sum_i (v^i)^2 = N\}$ for all $t\geq 0$, and the dynamics has $\sigma^N$, the uniform measure on $S^N$, as the unique stationary distribution. Kac worked with initial conditions $f_0^N$ having a density in $L^2(S^N,\sigma^N)$, for which we now know that $f_t^N$ equilibrates exponentially fast in the $L^2$ norm, with rates uniform in $N$, see \cite{carlen-carvalho-loss2000,janvresse2001}. However, the $L^2$ norm is a crude upper bound for the $L^1$ norm; moreover, the $L^2$ norm of typical initial distributions $f_0^N$ with a near-product structure (specifically, \emph{chaotic} sequences, see below) grows exponentially with $N$, which means that one has to wait a time proportional to $N$ in order for the $L^2$ bound to start providing evidence of convergence. Thus, one looks for alternative ways to quantify equilibration, such as convergence in relative entropy. The relative entropy of near-product measures grows linearly (and not exponentially) with $N$, which is a crucial advantage over the $L^2$ norm. The usual approach is to control the entropy production, in order to obtain an exponential rate of equilibration in relative entropy. Unfortunately, there exist sequences of initial distributions for which the entropy production degenerates as $N\to\infty$, as shown in \cite{einav2011}. It is worth noting, however, that the sequence constructed in \cite{einav2011} is physically unlikely, in the sense that $f_0^N$ gives half the total energy of the system to a small fraction of the particles. This raises the question of whether there is a smaller, but still rich, class of initial conditions for which one can have good control on the entropy production. We refer the reader to \cite{carlen-carvalho-leroux-loss-villani2010} for more details about Kac model and equilibration in relative entropy.

Picking up the challenge of choosing good (physical) initial conditions, and in order to avoid the badly behaved initial distributions for which entropy production degenerates, Bonetto, Loss, and Vaidyanathan \cite{bonetto-loss-vaidyanathan2014} introduced the model \eqref{eq:kac}-\eqref{eq:thermostat}, called the \emph{thermostated Kac particle system}, to describe a system in which all but a few particles are at equilibrium. This thermostated particle system no longer preserves the energy, so $f_t^N$ is supported on the whole space $\RR^N$, and the equilibrium distribution is the $N$-dimensional Gaussian with density $\gamma^{\otimes N}(\v) = \prod_i \gamma(v^i)$. In this case, the system approaches equilibrium in relative entropy exponentially fast, with rates uniform in the number of particles, see \cite[Theorem 3]{bonetto-loss-vaidyanathan2014}. Later, in \cite{bonetto-loss-tossounian-vaidyanathan2017}, the use of the ideal thermostat \eqref{eq:thermostat} was justified by approximating it with a finite but large reservoir of particles at equilibrium in a quantitative way.

\subsection{Propagation of chaos}

Besides the long-time behaviour of the particle system, one can also study convergence of $f_t^N$ as $N\to\infty$. Notice however that this is not an easy task, because even if we consider particles whose velocities are independent at $t=0$, the collisions amongst them will destroy this independence for later times. Nevertheless, for the thermostated Kac system, one expects the correlations between particles to become weaker as $N$ grows. The following concept formalizes this idea of asymptotic independence:

\begin{defi}[chaos]
For each $N\in\NN$, let $f^N$ be a symmetric probability measure on $\RR^N$. The collection $(f^N)_{N\in\NN}$ is said to be \emph{chaotic} with respect to some given probability measure $f$ on $\RR$, if for all $k\in\NN$, the marginal distribution of $f^N$ on the first $k$ variables converges in distribution, as $N\to\infty$, to the tensor product measure $f^{\otimes k}$. That is: for every $k\in\NN$ and every bounded and continuous function $\phi :\RR^k \to \RR$, it holds
\[
\lim_{N\to\infty} \int_{\RR^N} \phi(v^1,\ldots,v^k) f^N(d\v)
= \int_{\RR^k} \phi(v^1,\ldots,v^k) f(d v^1) \cdots f(dv^k).
\]
\end{defi}

For Kac's model, that is, when $\mu=0$, we know that if the sequence $( f_0^N )_{N\in \NN}$ is chaotic to some probability measure $f_0$ on $\RR$, then for all $t\geq 0$ the sequence $( f_t^N )_{N\in\NN}$ will also be chaotic to some $f_t$; this property is known as \emph{propagation of chaos}. The limit $f_t$ is the solution to the so-called \emph{Boltzmann-Kac equation}, which reads
\begin{equation}
\label{eq:Boltzmann-Kac}
\frac{d f_t}{dt}(v)
= 2\lambda \int_\RR \int_0^{2\pi} [f_t(v')f_t(v_*') - f_t(v)f_t(v_*)] \frac{d\theta}{2\pi} dv_*,
\end{equation}
in the case where $f_0$, and thus every $f_t$, has a density. This was first shown by Kac \cite{kac1956} in the special case where $f_t^N$ has a density in $L^2(S^N,\sigma^N)$. The solution to \eqref{eq:Boltzmann-Kac} also preserves the initial energy, i.e., $\int v^2 f_t(dv) = \int v^2 f_0(dv) = 1$ for all $t\geq 0$. It is straightforward to verify that the Gaussian density with energy 1 is a stationary distribution of the equation, and it is known that the solution converges to it, see for instance \cite{hauray2016,kac1956}.

When we introduce the thermostat to Kac's original model, propagation of chaos still holds, as shown in \cite[Theorem 5]{bonetto-loss-vaidyanathan2014}, and the limit density satisfies
\begin{equation}
\label{eq:thermostatedBoltzmannKac}
\begin{split}
\frac{d f_t}{dt}(v)
&= 2\lambda \int_\RR \int_0^{2\pi} [f_t(v')f_t(v_*') - f_t(v)f_t(v_*)] \frac{d\theta}{2\pi} dv_* \\
& \quad {} +\mu \int_\RR \int_0^{2\pi} [f_t(v') \gamma(v_*') - f_t(v) \gamma(v_*)] \frac{d\theta}{2\pi} dv_*,
\end{split}
\end{equation}
which we refer to as the \emph{thermostated Boltzmann-Kac equation}, or simply the \emph{kinetic equation}. As with the particle system, the solution to \eqref{eq:thermostatedBoltzmannKac} does not preserve the initial energy, and its equilibrium distribution is $\gamma$, the Gaussian density with energy $T$. When the initial condition $f_0$ has a density with finite relative entropy, it follows from \cite[Propostition 15]{bonetto-loss-vaidyanathan2014} that there is exponential convergence to equilibrium in relative entropy. In Definition \ref{def:weak} we will provide a notion of weak solution for \eqref{eq:thermostatedBoltzmannKac}, which will allow us to work with probability measures instead of densities. Using this notion, we give an existence and uniqueness result in Theorem \ref{thm:well_posedness}.



\subsection{Main result}

Chaoticity, and thus propagation of chaos, can be made quantitative. For Kac's model this was done in \cite{cortez2016} using Wasserstein distances, defined below, and providing explicit convergence rates in $N$ which are uniform in time. Similar quantitative results for the spatially homogeneous Boltzmann equation can be found for instance in \cite{cortez-fontbona2018,mischler-mouhot2013}.

The goal of the present article is to strengthen the propagation of chaos result for the thermostated Kac model in \cite{bonetto-loss-vaidyanathan2014}, by making it quantitative in $N$ with rates that are uniform in time. To quantify chaos we will use the following metric: given $f,g$ probability measures on $\RR^k$, their \emph{$2$-Wasserstein distance} is given by
\[
W_2(f,g)
= \left( \inf_{\mathbf{X},\mathbf{Y}} \EE\left[ \frac{1}{k} \sum_{i=1}^k (X^i - Y^i)^2 \right] \right)^{1/2},
\]
where the infimum is taken over all pairs of random vectors $\mathbf{X} = (X^1,\ldots,Y^k)$ and $\mathbf{Y} = (Y^1,\ldots,Y^k)$ such that $\Law(\mathbf{X}) = f$ and $\Law(\mathbf{Y}) = g$. This defines a distance in the space of probability measures with finite second moment. The infimum is always achieved by some $(\mathbf{X},\mathbf{Y})$, and such a pair is called an \emph{optimal coupling}; see \cite{villani2009} for details.

We will use the following characterization of chaoticity, see for instance \cite{sznitman1989}: a sequence $(f^N)_{N\in\NN}$ is $f$-chaotic if and only if for a sequence of random vectors $\mathbf{X}$ on $\RR^N$ with $\Law(\mathbf{X}) = f^N$, it holds that the sequence of random empirical measures
\[
\bar{\mathbf{X}}
:= \frac{1}{N} \sum_{i=1}^N \delta_{X^i}
\]
almost surely converges to the constant probability measure $f$. We can now state our main result.


\begin{thm}[uniform propagation of chaos]
\label{thm:UPoC}
Assume that $\int |v|^r f_0(dv) < \infty$ for some $r>4$. Let $(\V_t)_{t\geq 0}$ be the thermostated Kac $N$-particle system described by \eqref{eq:kac}-\eqref{eq:thermostat}, and let $(f_t)_{t\geq 0}$ be the unique weak solution of \eqref{eq:thermostatedBoltzmannKac}. Then there exists a constant $C$ depending only on $\lambda$, $\mu$, $T$, $r$, and $\int |v|^r f_0(dv)$, such that for all $t\geq 0$ we have:
\begin{equation}
\label{eq:UPoC}
    \EE[W_2^2(\bar{\V}_t, f_t)]
    \leq 4 e^{-\frac{\mu}{2} t} W_2^2(f_0^N, f_0^{\otimes N})
     + \frac{C}{N^{1/3}}.
\end{equation}
\end{thm}

This theorem gives a uniform-in-time propagation of chaos rate of order $N^{-1/3}$, provided that the first term $W_2^2(f_0^N,f_0^{\otimes N})$ converges to 0 at the same rate or faster; for instance, one can simply take $f_0^N = f_0^{\otimes N}$, so the first term vanishes. The rate $N^{-1/3}$ is not so far from the optimal rate $N^{-1/2}$, valid for the convergence of the empirical measure of an $N$-tuple of i.i.d.\ variables towards their common law, with the same metric as in \eqref{eq:UPoC}; see \cite[Theorem 1]{fournier-guillin2013}. We remark that if one only assumes $\int |v|^r f_0(dv) < \infty$ for some $2<r<4$, we can still deduce \eqref{eq:UPoC}, but with a slower chaos rate of order $N^{-\eta(r)}$ for some $0<\eta(r)<1/3$. We also remark that the value $\mu/2$, corresponding to the rate of decay of the initial condition term in \eqref{eq:UPoC} (see also the contraction estimates given in Lemma \ref{lem:contraction_ps} and Lemma \ref{lem:contraction_thermostatedBoltzmannKac} below),  coincides with the spectral gap of the generator of the particle system, and also with the bound on the entropy production obtained in \cite{bonetto-loss-vaidyanathan2014}; see also \cite{tossounian-vaidyanathan2015} for the optimality of this bound.

The proof of Theorem \ref{thm:UPoC} is based on a coupling argument developed in \cite{cortez-fontbona2016} and later used in \cite{cortez2016}. This argument makes use of a probabilistic object called the \emph{Boltzmann process}, which is a stochastic proces $(Z_t)_{t\geq 0}$ satisfying $\Law(Z_t) = f_t$ for all $t\geq 0$. More specifically, we will construct our particle system $\V_t = (V_t^1,\ldots,V_t^N)$ using a Poisson point measure, and couple it with a collection $\Z_t = (Z_t^1,\ldots,Z_t^N)$ of Boltzmann processes, in 
a way that the two remain close on expectation.

The structure of the article is as follows.
In Section \ref{sec:PDE} we provide a notion of weak solution for the thermostated Boltzmann-Kac equation \eqref{eq:thermostatedBoltzmannKac}, valid for collections $(f_t)_{t\geq 0}$ of probability measures, and we then prove a well-posedness result for this notion. In Section \ref{sec:coupling_construction} we specify the coupling construction mentioned above and we prove Theorem \ref{thm:UPoC}. Along the way, we will use this construction to prove some interesting results, such as the equilibration in $W_2$ for the particle system in Lemma \ref{lem:contraction_ps}, and an analogous result for the kinetic equation in Lemma \ref{lem:contraction_thermostatedBoltzmannKac}. Some final comments are given in Section \ref{sec:conclusion}.

\section{Well-posedness for the kinetic equation}
\label{sec:PDE}


In this section, we define a notion of weak solution to \eqref{eq:thermostatedBoltzmannKac}, and prove its well-posedness. We will not require each $f_t$ to have a density; instead, it will be an element of the space $\MM$ of bounded non-negative Borel measures on $\RR$ metrized by \emph{total variation} $\Vert \cdot \Vert$. We will see that, if $f_0$ is a probability measure, then $f_t$ will also be a probability measure for all $t>0$. Similarly, if $f_0$ has a density, so will $f_t$.

For convenience, let us introduce the mapping $B:\MM \times \MM \rightarrow \MM$, given by 
\begin{equation*}
    \int_\RR \phi(x) B[\nu_1, \nu_2](dx) = \int_\RR \int_\RR \nu_1(dx) \nu_1(dy) \int_0^{2\pi} \phi(x \cos\theta + y \sin\theta) \frac{d\theta}{2\pi}
\end{equation*}
for all bounded and continuous function $\phi$.
Notice that when $\nu_1$ and $\nu_2$ have densities $g_1$ and $g_2$ with respect to the Lebesgue measure, then $B[\nu_1, \nu_2]$, also denoted by $B[g_1,g_2]$, satisfies
\[
    B[g_1,g_2](v)
    = \int_0^{2\pi}\int_\RR g_1(v') g_2(v_*') dv_* \frac{d\theta}{2\pi}.
\]
We note that \eqref{eq:thermostatedBoltzmannKac} is equivalent to
\[
\frac{d f_t}{dt}
= 2\lambda (B[f_t,f_t] - f_t) + \mu (B[f_t,\gamma] - f_t).
\]
This motivates the following notion of weak solution:

\begin{defi}
\label{def:weak}
A function $f \in C([0,\infty), \MM)$ is a \emph{weak solution} to \eqref{eq:thermostatedBoltzmannKac} with initial condition $f_0$ if, for all $t\geq 0$, we have
\begin{equation}\label{eq:thermostatedBoltzmannKacweak0}
    f_t = f_0 +  \int_0^t \{ 2\lambda (B[f_s, f_s] -f_s)  + \mu (B[f_s, \gamma] - f_s) \} ds.
\end{equation}
\end{defi}

We summarize some of the useful properties of the mapping $B$ in the following lemma, which we state without proof.

\begin{lem}\label{lem:Bproperties}
\begin{enumerate}
\item Monotonicity:  If  $\nu_1, \nu_2, \pi_1$, and $\pi_2$ in $\MM$ are such that
\[
    \nu_1(A) \geq \pi_1(A) \text{ and } \nu_2(A) \geq \pi_2(A) \quad \text{$ \forall A$ measurable},
\]
then 
\[
    B[\nu_1, \nu_2](A) \geq B[\pi_1, \pi_2](A), \quad \text{$ \forall A$ measurable}.
\]

\item Norm: for all $\nu_1, \nu_1 \in \MM$, it holds
\[
    \Vert B[\nu_1, \nu_2] \Vert = \Vert \nu_1 \Vert \Vert \nu_2 \Vert.
\]
If $\nu_1$ and $\nu_2$ are bounded, signed, Borel measures, then
\[
    \Vert B[\nu_1, \nu_2] \Vert \leq \Vert \nu_1 \Vert \Vert \nu_2 \Vert.
\]
\item Second moments and arbitrary moments: If $\nu_1$ and $\nu_2$ in $\MM$ have finite second moments $e_1$ and $e_2$ respectively, then
\[
    \int_\RR x^2 B[\nu_1, \nu_2] (dx) = \frac{ e_2 + e_2 }{2}.
\]
If $\nu_1$ and $\nu_2$ have finite $r^{\text{th}}$ moments $n_r$ and $m_r$ for some $r > 0$, then
\[
    \int_\RR \vert x\vert^r B[\nu_1, \nu_2](dx) \leq  2^{\max\{\frac{r}{2},1\}} \frac{n_r + m_r}{2} \int_0^{2\pi} | \cos\theta |^r \frac{d\theta}{2\pi}.
 \]
\end{enumerate}
\end{lem}

We are now ready to state and prove our well-posednes result:

\begin{thm}[well-posedness]
\label{thm:well_posedness}
For every \emph{probability measure} $f_0 \in \MM$, there is a unique solution $f$ to \eqref{eq:thermostatedBoltzmannKacweak}. $f_t$ is a probability measure for every $t$. If $f_0$ has a density or a finite $r^{\text{th}}$ moment for some $r\geq 2$, then
 so does $f_t$ for all $t$.
\end{thm}

\begin{proof} We will use the following equivalent form of \eqref{eq:thermostatedBoltzmannKacweak0}:
\begin{equation}
\label{eq:thermostatedBoltzmannKacweak}
    f_t = e^{-(2\lambda+\mu) t} f_0 +  \int_0^t e^{-(2\lambda+\mu)(t-s)}\left( 2\lambda B[f_s, f_s]  + \mu B[f_s, \gamma] \right) ds.
\end{equation}
We use the iterative construction in \cite{tanaka.s1968}. Let $f_0$ be a Borel probability measure on $\RR$. Define the sub-probability measures $( u^n_t)_{n=0}^\infty$ inductively by
\begin{eqnarray}
    u^0_t & = & e^{-(2\lambda+\mu)t} f_0, \nonumber\\
    u^{n+1}_t & = & e^{-(2\lambda+\mu) t} f_0 + \int_0^t e^{-(2\lambda+\mu)(t-s)} \left( \mu B[ u^n_s, \gamma] + 2\lambda B[u^n_s, u^n_s] \right) ds \label{eq:iteration}
\end{eqnarray}
Using Lemma \ref{lem:Bproperties} we see that $u^n_t$ is continuous in $t$ for each $n$, that $u^n_t(\RR) \leq 1$, and that $(u^n_t)_n$ is increasing in $n$. Hence, for each $t$, $(u^n_t)_n$ converges to some element $u_t$ in $\MM$ and $u_t(\RR) \leq 1$. Note that $u_t-u_t^n$ is a non-negative measure for each $t$, thus we have convergence in total variation, since
\[
    \lim_{n\rightarrow \infty} \Vert u^n_t - u_t \Vert =  \lim_{n\rightarrow \infty}  u_t(\RR) - u^n_t(\RR) = 0.
\]
This, together with Lemma \ref{lem:Bproperties}, implies that 
\[
    \lim_{n\rightarrow \infty} \Vert B[u^n_t, \gamma] - B[u_t, \gamma] \Vert \leq \lim_{n\rightarrow \infty} \Vert u^n_t - u_t \Vert = 0,
\]
and
\[
     \lim_{n\rightarrow \infty} \Vert B[u^n_t, u^n_t] - B[u_t, u_t] \Vert \leq \lim_{n\rightarrow \infty} \Vert u^n_t - u_t \Vert (u^n_t(\RR) + u_t(\RR)) = 0.
\]
Thus we can take the infinite $n$ limit in \eqref{eq:iteration} and establish that $u_t$ solves \eqref{eq:thermostatedBoltzmannKacweak}. Being an increasing limit of continuous functions, $u: [0,\infty) \rightarrow \MM$ is lower semi-continuous, and thus measurable. Since $u_t(\RR) \leq 1$, $\forall t$, $u$ belongs to $L^\infty( [0,\infty), \MM)$. 
To show that $u_t$ is continuous (in $t$) we note that it equals
\[
    e^{-(2\lambda+\mu) t} f_0 + e^{-(2\lambda+\mu) t} \int_0^t e^{(2\lambda+\mu)s} \left( \mu B[ u_s, \gamma] + 2\lambda B[u_s, u_s] \right) ds
\]
and the integrand above is in the Bochner space $L^1([0,\tau], \MM )$ for all $\tau$. This makes $u_t$ continuous.
A consequence of this continuity is that:  $h(t)= u_t(\RR)$ is differentiable and satisfies the differential equation
\[
    h'(t) = -(2\lambda+\mu) h(t) + \mu h(t) + 2\lambda h(t)^2
\]
Since $h(0)=1$, $h(t) \equiv 1$. Hence, $u_t$ is a probability measure for all $t$.
To show the uniqueness of $u_t$, let $g_t \in C([0,\infty), \MM)$ satisfy \eqref{eq:thermostatedBoltzmannKacweak}. 
On one hand, $g_t\geq u^0_t$ by definition. And thus, by induction, $g_t \geq u^n_t$ a.e. $t$ for all $n$. By the monotone convergence theorem, we have
\begin{equation*}
    g_t(A) \geq u_t(A)
\end{equation*}
for every measurable set $A$. On the other hand, using Lemma \ref{lem:Bproperties}, for each $t$ we obtain
\[
    \int_0^t e^{-(2\lambda+\mu)(t-s)} \Vert (\mu B[g(s),\gamma] + 2\lambda B[g_s, g_s]) \Vert ds \leq  \mu \sqrt{t} \left(\int_0^t \Vert g_s \Vert^2 ds\right)^{\frac{1}{2}} + 2\lambda \int_0^t \Vert g_s\Vert^2 ds 
\]
which shows that $g_t$ is continuous in $t$, and just like $u_t$, must be a probability measure for all $t$. Thus, $\Vert g_t - u_t \Vert = g_t(\RR)- u_t(\RR) = 1 - 1 =0$.

To prove the last statement of the theorem, we note that if $f_0 \in L^1(\RR)$, then $u_t^n \in L^1(\RR)$ for all $\RR$ and we use the completeness of $L^1$ under the total variation norm. 
If $f_0$ has a finite $r^{\text{th}}$ moment for some $r>0$, then by Lemma \ref{lem:Bproperties} and induction, we see that, for each $t$, $(\int_\RR u_t^n(dv) \vert v\vert^r)_n$ is finite, monotone increasing, and bounded above by the solution $R(t)$ to the following integral equation:
\[
    R(t) = e^{-(2\lambda + \mu) t} \int_\RR \vert v\vert^r f_0(dv) + C_r\int_0^t e^{-(2\lambda+\mu)(t-s)} (2\lambda+ \frac{\mu}{2})R(s) ds + C_r \int_\RR \vert w\vert^r \gamma(dw).
\]
Here $C_r = 2^{\max\{\frac{r}{2},1\}}\int_0^{2\pi} \vert \cos\theta\vert^r \frac{d\theta}{2\pi}$ is as in Lemma \ref{lem:Bproperties}. $R(t)$ is finite due to Gronwall's inequality. 
The monotone convergence theorem implies that $R(t)$ controls the $r^{\text{th}}$ moment of $f_t$.
\end{proof}

It is straightforward to verify that, if $\int_\RR v^2 f_0(dv)<\infty$, then we have
\begin{equation}
\label{eq:second-moment}
\int_\RR v^2 f_t(dv) =
    \left( \int_\RR v^2 f_0(dv) \right) e^{-\frac{1}{2} \mu t } + T (1-e^{-\frac{1}{2} \mu t }).
 \end{equation}

\begin{rmk} The uniqueness of the solution to \eqref{eq:thermostatedBoltzmannKacweak0} holds in the larger space $L^2_\text{loc}([0, \infty), \MM)$, provided we identify functions $f_t$ that agree $t$-a.s.
\end{rmk}

\section{Coupling construction}
\label{sec:coupling_construction}

\subsection{Particle system}\label{subsec:particle-system}

We provide an explicit construction of the particle system using an SDE, following \cite{cortez-fontbona2016}. To this end, for fixed $N\in\NN$, let $\R(dt,d\theta,d\xi,d\zeta)$ be a Poisson point measure on $[0,\infty)\times[0,2\pi)\times [0,N)^2$ with intensity
\[
N\lambda dt \frac{d\theta}{2\pi} \frac{d\xi d\zeta \ind_{\{\ii(\xi) \neq \ii(\zeta)\}}}{N(N-1)}
= \frac{\lambda dt d\theta d\xi d\zeta \ind_{\{\ii(\xi) \neq \ii(\zeta)\}}}{2\pi(N-1)},
\]
where $\ii$ is the function that associates to a variable $\xi\in[0,N)$ the discrete index $\ii(\xi) = \lfloor \xi \rfloor +1 \in \{1,\ldots,N\}$. In words: at rate $N\lambda$, the measure $\R$ selects collision times $t\geq 0$, and for each such time, it independently samples a parameter $\theta$ uniformly at random on $[0,2\pi)$, and a pair $(\xi,\zeta) \in [0,N)^2$ such that $\ii(\xi) \neq \ii(\zeta)$, also uniformly. The pair $(\ii(\xi),\ii(\zeta))$ provides the indices of the particles involved in Kac-type collisions. The fact that we use continuous variables $\xi,\zeta \in [0,N)$, instead of discrete indices in $\{1,\ldots,N\}$, will be crucial to define our coupling with a collection of Boltzmann processes.

Let $\Q_1(dt, d\theta, dw),\ldots,\Q_N(dt, d\theta, dw)$ be a collection of independent Poisson point measures on $[0,\infty)\times [0,2\pi) \times \RR$, also independent of $\R$, each having intensity $\mu dt \frac{d\theta}{2\pi} \gamma(dw)$. Finally, let $\V_0 = (V_0^1,\ldots,V_0^N)$ be an exchangeable collection of random variables with $\Law(\V_0) = f_0^N$, independent of everything else.

The particle system $\V_t = (V_t^1,\ldots,V_t^N)$ is defined as the unique jump-by-jump solution of the SDE
\begin{equation}
    \label{eq:sde_ps}
    \begin{split}
        d\V_t
        &= \int_0^{2\pi} \int_{[0,N)^2} \sum_{i,j=1, i\neq j}^N [ \mathbf{a}_{ij}(\V_{t^\-},\theta) - \V_{t^\-} ] \ind_{\{ \ii(\xi)=i,\ii(\zeta)=j\}} \R(dt,d\theta,d\xi,d\zeta) \\
        &\quad {} + \sum_{i=1}^N \int_0^{2\pi} \int_\RR [\mathbf{b}_i(\V_{t^\-},\theta, w) - \V_{t^\-}] \Q_i(dt, d\theta, dw)
    \end{split}
\end{equation}
that starts at $\V_0$.
Here, for $\v \in \RR^N$, the vectors $\mathbf{a}_{ij}(\v,\theta)\in\RR^N$ and $\mathbf{b}_i(\v,\theta,w)\in\RR^N$ are defined as
\begin{align*}
\mathbf{a}_{ij}(\v,\theta)^k
&= \begin{cases}
v^i \cos\theta - v^j \sin\theta & \text{if $k=i$,} \\
v^i \sin\theta + v^j \cos\theta & \text{if $k=j$,} \\
v^k & \text{otherwise},
\end{cases} \\
\mathbf{b}_i(\v,\theta,w)^k
&= \begin{cases}
v^i\cos\theta - w \sin\theta & \text{if $k=i$,} \\
v^k & \text{otherwise}.
\end{cases}
\end{align*}
For any $i=1,\ldots,N$, from \eqref{eq:sde_ps} it follows that particle $V_t^i$ satisfies the SDE 
\begin{equation}
\label{eq:sde_vi}
    \begin{split}
        dV_t^i
        &= \int_0^{2\pi} \int_0^N [V_{t^\-}^i \cos\theta - V_{t^\-}^{\ii(\xi)}\sin\theta - V_{t^\-}^i] \P_i(dt,d\theta,d\xi) \\
        & \quad {} + \int_0^{2\pi} \int_\RR [V_{t^\-}^i \cos\theta - w\sin\theta - V_{t^\-}^i] \Q_i(dt, d\theta,dw),
    \end{split}
\end{equation}
where $\P_i$ is defined as
\[
\P_i(dt,d\theta,d\xi)
= \R(dt,d\theta,[i-1,i), d\xi)
+ \R(dt,-d\theta, d\xi, [i-1,i)),
\]
and where we use $-d\theta$ to transform $\sin\theta$ into $-\sin\theta$. Clearly, $\P_i$ is a Poisson point measure on $[0,\infty)\times[0,2\pi)\times[0,N)$ with intensity
\[
2 \lambda dt \frac{d\theta}{2\pi} \frac{d\xi \ind_{\{\ii(\xi)\neq i\}}}{N-1}.
\]

As mentioned earlier,  $f_t^N = \Law(\V_t)$ converges exponentially fast to the Gaussian density $\gamma^{\otimes N}$ in relative entropy, as shown in \cite[Theorem 3]{bonetto-loss-vaidyanathan2014}. Similarly, the following result provides equilibration in $W_2$, which does not require $f_t^N$ to have a density:

\begin{lem}[contraction and equilibration for the particle system]
\label{lem:contraction_ps}
Let $f_t^N$ and $\tilde{f}_t^N$ be the laws of the thermostated Kac $N$-particle systems starting from (possibly different) symmetric initial distributions $f_0^N$ and $\tilde{f}_0^N$, respectively. Then
\[
W_2^2(f_t^N , \tilde{f}_t^N)
\leq e^{-\frac{\mu}{2} t} W_2^2(f_0^N , \tilde{f}_0^N).
\]
Consequently, taking $\tilde{f}_0^N$ as the stationary distribution $\gamma^{\otimes N}$, gives
\[
W_2^2(f_t^N , \gamma^{\otimes N})
\leq e^{-\frac{\mu}{2} t} W_2^2(f_0^N , \gamma^{\otimes N}).
\]

\end{lem}

\begin{proof}
Let $(\V_t)_{t\geq 0}$ and $(\tilde{\V}_t)_{t\geq 0}$ be the solutions to the SDE \eqref{eq:sde_ps} with respect to the same Poisson point measures $\R$, $\Q_1,\ldots,\Q_N$, but starting from initial conditions $(\V_0,\tilde{\V}_0)$ which we take as an optimal coupling between $f_0^N$ and $\tilde{f}_0^N$. Call $h(t) = \EE[(V_t^1 - \tilde{V}_t^1)^2]$, then $W_2^2(f_t^N,\tilde{f}_t^N) \leq \EE[\frac{1}{N} \sum_i (V_t^i - \tilde{V}_t^i)^2 ] = h(t)$ by exchangeability, with equality at $t=0$. Thus, it suffices to study $h(t)$. Since both $V_t^1$ and $\tilde{V}_t^1$ satisfy \eqref{eq:sde_vi} with $i=1$, when computing the increments of $(V_t^1-\tilde{V}_t^1)^2$ the terms $w \sin\theta$ cancel, thus obtaining
\begin{align}
    \notag
    h'(t)
    &= 2\lambda \EE \int_0^{2\pi} \int_1^N
    \left[ (V_t^1 \cos\theta - V_t^{\ii(\xi)}\sin\theta - \tilde{V}_t^1 \cos\theta + \tilde{V}_t^{\ii(\xi)} \sin\theta)^2 \right. \\
    \notag
    & \qquad \qquad \qquad \qquad \left. {}- (V_t^1 - \tilde{V}_t^1)^2 \right] \frac{d\theta d\xi}{2\pi (N-1)} \\
    \notag
    & \qquad {} + \mu \EE \int_0^{2\pi} \int_\RR \left[ (V_t^1 \cos\theta - \tilde{V}_t^1 \cos\theta)^2 - (V_t^1 - \tilde{V}_t^1)^2 \right] \frac{d\theta \gamma(dw)}{2\pi} \\
    \label{eq:dht_contraction}
    &= 2\lambda \EE \int_1^N
    \left[ -\frac{1}{2}(V_t^1 - \tilde{V}_t^1)^2  +   \frac{1}{2}(V_t^{\ii(\xi)} -  \tilde{V}_t^{\ii(\xi)})^2 \right] \frac{d\xi}{ N-1}
    - \frac{\mu}{2} h(t),
\end{align}
where we used that $\int_0^{2\pi} \cos^2\theta \frac{d\theta}{2\pi} = \frac{1}{2} = \int_0^{2\pi} \sin^2\theta \frac{d\theta}{2\pi}$ and $\int_0^{2\pi} \cos\theta\sin\theta \frac{d\theta}{2\pi} = 0$. Notice that
\[
\EE \int_1^N (V_t^{\ii(\xi)} -  \tilde{V}_t^{\ii(\xi)})^2 \frac{d\xi}{N-1}
= \EE \frac{1}{N-1} \sum_{i=2}^N (V_t^i - \tilde{V}_t^i)^2
= h(t),
\]
thus the first term in \eqref{eq:dht_contraction} vanishes, which then gives $h'(t) = -\frac{\mu}{2} h(t)$. The desired bound follows.
\end{proof}

\subsection{Coupling with Boltzmann processes}

For a given probability measure $f_0$, let $(f_t)_{t\geq 0}$ be the unique weak solution of \eqref{eq:thermostatedBoltzmannKac} given by Theorem \ref{thm:well_posedness}. We will now construct a stochastic process $(Z_t)_{t\geq 0}$, called the \emph{Boltzmann process}, such that $\Law(Z_t) = f_t$ for all $t\geq 0$. This process is the probabilistic counterpart of \eqref{eq:thermostatedBoltzmannKac}, and it represents the trajectory of a single particle immersed in the infinite population. It was first introduced by Tanaka \cite{tanaka1978} in the context of the Boltzmann equation for Maxwell molecules.

Consider a Poisson point measure $\P(dt,d\theta,dz)$ on $[0,\infty) \times [0,2\pi) \times \RR $ with intensity $2\lambda dt \frac{d\theta}{2\pi} f_t(dz)$, and an independent Poisson point measure $\Q(dt,d\theta,dw)$ on $[0,\infty) \times [0,2\pi) \times \RR$ with intensity $\mu dt \frac{d\theta}{2\pi} \gamma(dw)$. Consider also a random variable $Z_0$ with law $f_0$, independent of $\P$ and $\Q$. The process $Z_t$ is defined as the unique solution, starting from $Z_0$, to the stochastic differential equation
\begin{equation}
\label{eq:sde_boltzmann_process}
\begin{split}
dZ_t
&= \int_0^{2\pi} \int_\RR [Z_{t^\-} \cos\theta - z\sin\theta - Z_{t^\-}] \P(dt,d\theta,dz) \\
& \quad {} + \int_0^{2\pi} \int_\RR [Z_{t^\-}\cos\theta - w\sin\theta - Z_{t^\-}] \Q(dt,d\theta,dw).
\end{split}
\end{equation}
Strong existence and uniqueness of solutions for this SDE is straightforward, since the rates of $\P$ and $\Q$ are finite on bounded time intervals. To show that $\Law(Z_t) = f_t$, the argument is classical: one first shows that $\ell_t := \Law(Z_t)$ solves
\[
    \ell_t
    =  f_0 + \int_0^t \{ 2 \lambda (B[\ell_s, f_s] - \ell_s)
    + \mu ( B[\ell_s, \gamma] - \ell_s)\} ds,
\]
which is a linearized version of \eqref{eq:thermostatedBoltzmannKacweak0}. This equation has a unique solution in the space $C([0, \infty), \MM)$ because the mapping $\nu\mapsto B[\nu,f_t]$ is non-expanding in total variation for all $t$. Since $f_t$ is a solution of this linearized version, we must have that $\ell_t = f_t$.


Since $\Law(Z_t) = f_t$, we can thus use the Boltzmann process as a tool to prove properties of the solution of the thermostated Boltzmann-Kac equation \eqref{eq:thermostatedBoltzmannKac}. For instance, we have the following lemma, which will be needed later to prove our uniform-in-time propagation of chaos result.

\begin{lem}[propagation of moments]
\label{lem:moments} Let $(f_t)_{t\geq 0}$ be the weak solution to \eqref{eq:thermostatedBoltzmannKac}. Let $r\geq 2$, and assume that  $\int_\RR \vert v \vert^r f_0(dv)< \infty$. Then $\sup_{t \geq 0} \int_\RR \vert v\vert^r f_t(dv)<\infty$.
\end{lem}

\begin{proof}
The case $r=2$ follows from \eqref{eq:second-moment}, so we assume $r>2$. Let $(Z_t)_{t\geq 0}$ be the Boltzmann process, i.e., the solution to \eqref{eq:sde_boltzmann_process}. Let $h(t) = \EE \vert Z_t \vert^r = \int_\RR \vert v\vert^r f_t(dv)$. We know from Theorem \ref{thm:well_posedness} that $h(t)<\infty$ for all $t$. Then $h(t)$ satisfies
\begin{align*}
    h'(t)
    &= 2\lambda \EE \int_{0}^{2\pi} \frac{d\theta}{2\pi} \int_{\RR}  f_t(dz) \left( \vert Z_{t^-} \cos\theta - z \sin\theta \vert^r - \vert Z_{t}\vert^r \right) \\
    & \quad {} +  \mu \EE \int_{0}^{2\pi} \frac{d\theta}{2\pi} \int_\RR \gamma(dw)\left( \vert Z_{t^-} \cos\theta - w \sin\theta \vert^r - \vert Z_{t}\vert^r \right).
\end{align*}
Note that $\EE|Z_t|^{r-1} \leq h(t)^{1-1/r}$ and $\EE|Z_t| \leq \max\{T, \int_\RR v^2 f_t(dv)\}^{1/2}$, thanks to \eqref{eq:second-moment} and Jensen's inequality. Using the inequality $(a+b)^r \leq a^r + b^r + 2^{r-1} (a b^{r-1} +a^{r-1} b)$ valid for $a,b \geq 0$, we thus obtain
\begin{equation}
\label{eq:moment_iniq}
    h'(t) \leq -C_1 h(t) + C_2  + C_3 h(t)^{1-1/r},
\end{equation}
where 
\[
 C_1 = 2\lambda \left(1 - 2\int_0^{2\pi} \vert \cos\theta \vert^r \frac{d\theta}{2\pi}\right) + \mu \left(1 - \int_0^{2\pi} \vert \cos\theta \vert^r \frac{d\theta}{2\pi} \right) >0 ,
\]
and $C_2, C_3>0$ are constants depending on $\lambda$, $\mu$, $r$, $T$, $\int_\RR v^2 f_t(dv)$, and some moments of $\gamma$ of order at most $r$. The statement follows from \eqref{eq:moment_iniq}.
\end{proof}

The Boltzmann process \eqref{eq:sde_boltzmann_process} is particularly useful in coupling arguments, as the next result shows. It provides contraction for the thermostated Boltzmann-Kac equation in $W_2$-distance:

\begin{lem}[contraction and equilibration for the thermostated Boltzmann-Kac equation]
\label{lem:contraction_thermostatedBoltzmannKac}
Let $f_t$, $\tilde{f}_t$ be the weak solutions to \eqref{eq:thermostatedBoltzmannKac} starting from some possibly different probability measures $f_0$, $\tilde{f}_0$. Then
\[
W_2^2(f_t,\tilde{f}_t)
\leq e^{-\frac{\mu}{2}t} W_2^2(f_0,\tilde{f}_0).
\]
Consequently, taking $f_0 = \gamma$, gives
\[
W_2^2(f_t,\gamma)
\leq e^{-\frac{\mu}{2}t} W_2^2(f_0, \gamma).
\]
\end{lem}

\begin{proof}
	For all $t\geq 0$, let $\Pi_t$ be an optimal coupling between $f_t$ and $\tilde{f}_t$, that is, $\Pi_t$ is a probability measure on $\RR\times\RR$ such that $\int (z-\tilde{z})^2 \Pi_t(dz,d\tilde{z}) = W_2^2(f_t,\tilde{f}_t)$. Let $\mathcal{S}(dt,d\theta,dz,d\tilde{z})$ be a Poisson point measure on $[0,\infty)\times[0,2\pi)\times\RR\times\RR$ with intensity $2\lambda dt \frac{d\theta}{2\pi}\Pi_t(dz,d\tilde{z})$, and define $\P(dt,d\theta,dz) = \mathcal{S}(dt,d\theta,dz,\RR)$ and $\tilde{\P}(dt,d\theta,d\tilde{z}) = \mathcal{S}(dt,d\theta,\RR,d\tilde{z})$. In words, $\P$ and $\tilde{\P}$ are Poisson point measures, with intensities $2\lambda dt \frac{d\theta}{2\pi} f_t(dz)$ and $2\lambda dt \frac{d\theta}{2\pi} \tilde{f}_t(d\tilde{z})$ respectively, which have the same atoms in the $t$ and $\theta$ variables, and with optimally-coupled realizations of $f_t$ and $\tilde{f}_t$ on the $z$ and $\tilde{z}$ variables. Also, let $\Q(dt,dw)$ be a Poisson point measure with intensity $\mu dt \gamma(dw)$ that is independent of $\mathcal{S}$, and set $\tilde{\Q} = \Q$. Let also $(Z_0, \tilde{Z}_0)$ be a realization of $\Pi_0$, independent of everything else; in particular we have $\EE[(Z_0-\tilde{Z}_0)^2] = W_2^2(f_0,\tilde{f}_0)$.
		
	Let $Z_t$ and $\tilde{Z}_t$ be the solutions to the SDE \eqref{eq:sde_boltzmann_process} with respect to $(\P,\Q)$ and $(\tilde{\P},\tilde{\Q})$, respectively, thus $\Law(Z_t) = f_t$ and $\Law(\tilde{Z}_t) = \tilde{f}_t$. Consequently, we have $W_2^2(f_t,\tilde{f}_t) \leq \EE[(Z_t - \tilde{Z}_t)^2] =: h(t)$.
	Using It\^{o} calculus, we have:
	\begin{align*}
	    h'(t)
	    &= 2\lambda \EE \int_0^{2\pi} \int_{\RR\times\RR} [(Z_t\cos\theta - z\sin\theta - \tilde{Z}_t\cos\theta + \tilde{z}\sin\theta)^2 - (Z_t-\tilde{Z}_t)^2] \Pi_t(dz,d\tilde{z}) \frac{d\theta}{2\pi} \\
	    & \quad {} + \mu \EE \int_0^{2\pi} \int_\RR [(Z_t \cos\theta - w\sin\theta - \tilde{Z}_t \cos\theta + w \sin\theta)^2 - (Z_t-\tilde{Z}_t)^2] \gamma(dw) \frac{d\theta}{2\pi} \\
	    &= 2\lambda \EE \int_0^{2\pi} \int_{\RR\times\RR} [(\cos^2\theta - 1)(Z_t-\tilde{Z}_t)^2 + (z-\tilde{z})^2\sin^2\theta ] \frac{d\theta}{2\pi} \Pi_t(dz,d\tilde{z})
	    - \frac{\mu}{2} h(t),
	\end{align*}
	where in the last step the cross term vanished because $\int_0^{2\pi} \cos\theta\sin\theta d\theta = 0$. Since $\int(z-\tilde{z})^2 \Pi_t(dz,d\tilde{z}) = W_2^2(f_t,\tilde{f}_t) \leq h(t)$, the integral in the last line is bounded above by 0. We thus obtain $h'(t) \leq - \frac{\mu}{2} h(t)$, which yields the result.
\end{proof}

We now specify the coupling construction that will allow us to prove our main result. We closely follow \cite{cortez-fontbona2016}, see also \cite{cortez2016}. The key idea is to define a system $\Z_t = (Z_t^1,\ldots,Z_t^N)$ of Boltzmann processes such that, for each $i=1,\ldots,N$, the process $Z_t^i$ mimics as closely as possible the dynamics of particle $V_t^i$. Comparing \eqref{eq:sde_vi} and \eqref{eq:sde_boltzmann_process}, we see that a way of achieving this is to define $Z_t^i$ as the solution of \eqref{eq:sde_vi}, but replacing $V_{t^\-}^{\ii(\xi)}$, which is a $\xi$-realization of the (random) empirical measure $\frac{1}{N-1} \sum_{j\neq i} \delta_{V_{t^\-}^j}$, with a $\xi$-realization of $f_t$. Moreover, we will do this in an \emph{optimal} way.

Specifically: we define $Z_t^i$ as the unique jump-by-jump solution to
\begin{equation}
\label{eq:sde_zi}
    \begin{split}
        dZ_t^i
        &= \int_0^{2\pi} \int_0^N [Z_{t^\-}^i \cos\theta - F_t^i(\Z_{t^\-},\xi)\sin\theta - Z_{t^\-}^i] \P_i(dt,d\theta,d\xi) \\
        & \quad {} + \int_0^{2\pi} \int_\RR [Z_{t^\-}^i \cos\theta - w\sin\theta - Z_{t^\-}^i] \Q_i(dt, d\theta,dw),
    \end{split}
\end{equation}
where we have used the same Poisson point measures $\P_i$ and $\Q_i$ as in \eqref{eq:sde_vi}. Here, $F^i$ is a measurable function $[0,\infty) \times \RR^N \times [0,N) \ni (t,\z,\xi) \mapsto F_t^i(\z,\xi) \in \RR$ with the following property: for any $t\geq 0$, $\z\in\RR^N$, and any random variable $U$ uniformly distributed on the set $[0,N) \backslash [i-1,i)$, the pair $(z^{\ii(U)}, F_t^i(\z, U))$ is an optimal coupling between the empirical measure $\bar{\z}^i := \frac{1}{N-1} \sum_{j\neq i} \delta_{z^j}$ and $f_t$. In other words,
\begin{equation}
    \label{eq:EW2ftZi}
\int_0^N \left(z^{\ii(\xi)} - F_t^i(\z,\xi) \right)^2 \frac{d\xi \ind_{\{\ii(\xi)\neq i\}}}{N-1}
= W_2^2(\bar{\z}^i, f_t).
\end{equation}
(The values of $F_t^i(\z,\xi)$ for $\xi\in[i-1,1)$ are irrelevant). We refer the reader to \cite[Lemma 3]{cortez-fontbona2016} for a proof of existence of such a function. The same result also ensures that $F_t^i$ satisfies the following: for any exchangeable random vector $\mathbf{X}$ in $\RR^N$, and any measurable function $\phi$, one has for $j\neq i$
\begin{equation}
    \label{eq:E_intj_phi_Fti}
    \EE \int_{j-1}^j \phi(F_t^i(\mathbf{X},\xi)) d\xi
    = \int_\RR \phi(v) f_t(dv).
\end{equation}

We take an initial condition $\Z_0 = (Z_0^1,\ldots,Z_0^N)$ with distribution $f_0^{\otimes N}$ and optimally coupled to $\V_0$, thus
\begin{equation}\label{eq:W2f0N}
\EE[ (V_0^1 - Z_0^1)^2 ]
= \EE\left[ \frac{1}{N} \sum_{i=1}^N (V_0^i - Z_0^i)^2 \right]
= W_2^2( f_0^N, f_0^{\otimes N}),
\end{equation}
by exchangeability. We have thus defined a collection $\Z_t = (Z_t^1,\ldots,Z_t^N)$, where each $Z_t^i$ is a Boltzmann process by construction; in particular, we have $\Law(Z_t^i) = f_t$. However, notice that $Z_t^i$ and $Z_t^j$ have a simultaneous jump whenever $V_t^i$ and $V_t^j$ undergo a Kac collision, which implies that $Z_t^i$ and $Z_t^j$ are \emph{not independent}. In order for this construction to be useful, one needs to prove that these Boltzmann processes become asymptotically independent as $N\to\infty$, as is done in \cite{cortez2016,cortez-fontbona2016}. This is the content of the following lemma, which moreover provides explicit rates in $N$, uniformly on time:

\begin{lem}[decoupling of Boltzmann processes]\label{lem:decoupling}
There exists a constant $C<\infty$ depending only on $\lambda$, $\mu$, $T$, and $\int v^2 f_0(dv)$, such that for all fixed $k\in\NN$ we have for all $t\geq 0$:
\[
W_2^2\left(\Law(Z_t^1,\ldots,Z_t^k), f_t^{\otimes k} \right)^2
\leq \frac{Ck}{N}.
\]
\end{lem}

\begin{proof}
The argument is the same as in \cite[Lemma 6]{cortez-fontbona2016} and \cite[Lemma 3]{cortez2016}, so we only provide the main steps of the proof here. The idea is to again use a coupling argument: for fixed $k \leq N$, we will define $k$ \emph{independent} Boltzmann processes $\tilde{Z}_t^1,\ldots,\tilde{Z}_t^k$ that remain close to $Z_t^1,\ldots,Z_t^k$ on expectation. To achieve this, each $\tilde{Z}_t^i$ will use the same randomness that defines $Z_t^i$ (i.e., the SDE \eqref{eq:sde_zi}), except when $Z_t^i$ has a simultaneous jump with $Z_t^j$ for some $j\in\{1,\ldots,k\}$, in which case either $\tilde{Z}_t^i$ or $\tilde{Z}_t^j$ will not jump.
To compensate for the missing jumps, we will use an additional independent source of randomness to define new jumps. Since on expectation this occurs only a proportion $k/N$ of the jumps of the collection $Z_t^1,\ldots,Z_t^k$, this construction will give the desired estimate.

To this end, let $\tilde{\R}$ be an independent copy of the Poisson point measure $\R$ introduced at the beginning of Section \ref{subsec:particle-system}, and for $i=1,\ldots,k$, define
\begin{align*}
\tilde{\P}_i(dt,d\theta,d\xi)
&= \R(dt,d\theta,[i-1,i), d\xi) \\
& \qquad {} + \R(dt,-d\theta, d\xi, [i-1,i)) \ind_{[k,N)}(\xi) \\
& \qquad {} + \tilde{\R}(dt, -d\theta, d\xi, [i-1,i)) \ind_{[0,k)}(\xi),
\end{align*}
which is a Poisson point measure with intensity $2\lambda dt d\theta d\xi \ind_{\{\ii(\xi)\neq i\}}/ [2\pi(N-1)]$, just as $\P_i$. Note that the Poisson measures $\tilde{\P}_1,\ldots,\tilde{\P}_k$ are independent by construction. Mimicking \eqref{eq:sde_zi}, we define $\tilde{Z}_t^i$ as the solution, starting from $\tilde{Z}_0^i = Z_0^i$, to the SDE
\begin{equation}
    \begin{split}
    \label{eq:SDE_tilde_zi}
    d\tilde{Z}_t^i
    &= \int_0^{2\pi} \int_0^N [\tilde{Z}_{t^\-}^i \cos\theta - F_t^i(\mathbf{Z}_{t^\-},\xi)\sin\theta - \tilde{Z}_{t^\-}^i] \tilde{\P}_i(dt,d\theta,d\xi) \\
    & \quad {} + \int_0^{2\pi} \int_\RR [\tilde{Z}_{t^\-}^i \cos\theta - w\sin\theta - \tilde{Z}_{t^\-}^i] \Q_i(dt, d\theta,dw).
    \end{split}
\end{equation}
It is clear that $\tilde{Z}_t^1,\ldots,\tilde{Z}_t^k$ is an exchangeable collection of Boltzmann processes. Moreover, using the independence of $\tilde{\P}_1,\ldots,\tilde{\P}_k$ and the fact that $F_t^i(\mathbf{z},\xi)$ has distribution $f_t$ for any $\mathbf{z}\in \RR^N$ and any $\xi$ uniformly distributed on $[0,N)\backslash [i-1,i)$, one can prove that the processes $\tilde{Z}_t^1,\ldots,\tilde{Z}_t^k$ are independent. For a full proof of this fact in a very similar setting, we refer the reader to \cite[Lemma 6]{cortez-fontbona2016}.

Call $h(t) := \EE[(Z_t^1 - \tilde{Z}_t^1)^2]$. By exchangeability, we have
\[
W_2^2\left(\Law(Z_t^1,\ldots,Z_t^k) , f_t^{\otimes k} \right)
\leq \EE\left[ \frac{1}{k} \sum_{i=1}^k (Z_t^i - \tilde{Z}_t^i)^2\right]
= h(t),
\]
thus it suffices to obtain the desired estimate for $h(t)$. From \eqref{eq:sde_zi} and \eqref{eq:SDE_tilde_zi}, using It\^o calculus, we obtain:
\begin{equation}
    \begin{split}
    \label{eq:dht_decoupling}
    h'(t)
    &= \EE \int_0^{2\pi} \int_0^N \Delta_1 \left[\R(dt,d\theta,[0,1), d\xi) + \R(dt,-d\theta,d\xi,[0,1))\ind_{[k,N)}(\xi)\right] \\
    & \qquad {} + \EE \int_0^{2\pi} \int_0^N \Delta_2  \R(dt,-d\theta,d\xi,[0,1))\ind_{[0,k)}(\xi) \\
    & \qquad {} + \EE \int_0^{2\pi} \int_0^N \Delta_3  \tilde{\R}(dt,-d\theta,d\xi,[0,1))\ind_{[0,k)}(\xi) \\
    & \qquad {} + \EE \int_0^{2\pi} \int_\RR \Delta_4 \Q_1(dt,d\theta,dw),
    \end{split}
\end{equation}
where $\Delta_1$ corresponds to the increment of $(Z_t^1 - \tilde{Z}_t^1)^2$ when $Z_t^1$ and $\tilde{Z}_t^1$ have a simultaneous Kac-type jump, $\Delta_2$ is the increment when only $Z_t^1$ jumps, $\Delta_3$ is the increment when only $\tilde{Z}_t^1$ jumps, and $\Delta_4$ is the increment when there is a thermostat interaction.

Thanks to the indicator $\ind_{[0,k)}(\xi)$, the fact that $\Delta_2$ and $\Delta_3$ involve only second-order products of $f_t$-distributed variables, using \eqref{eq:E_intj_phi_Fti}, and recalling that \eqref{eq:second-moment} implies that $\int v^2 f_t(dv) \leq \max \left\{ \int v^2 f_0(dv) , T \right\}$, we deduce that the second and third terms in \eqref{eq:dht_decoupling} are bounded above by $\frac{Ck}{N}$. On the other hand, since the term $F_t^i(\mathbf{Z}_{t^\-},\xi)$ appears in both \eqref{eq:sde_zi} and \eqref{eq:SDE_tilde_zi}, it will cancel out in $\Delta_1$; more specifically, we have
\begin{align*}
    \Delta_1 &= \left[Z_{t^\-}^1 \cos\theta - F_t^1(\mathbf{Z}_{t^\-},\xi) \sin\theta
    -\tilde{Z}_{t^\-}^1 \cos\theta + F_t^1(\mathbf{Z}_{t^\-},\xi) \sin\theta
    \right]^2
    - (Z_{t^\-}^1 - \tilde{Z}_{t^\-}^1)^2 \\
    &= - (1-\cos^2\theta) (Z_{t^\-}^1 - \tilde{Z}_{t^\-}^1)^2
    \leq 0.
\end{align*}
Similarly, it can be easily seen that $\Delta_4 = -(1-\cos^2\theta) (Z_{t^\-}^1-\tilde{Z}_{t^\-}^1)^2$,
then the last term in \eqref{eq:dht_decoupling} is equal to $-\frac{\mu}{2} h(t)$. Thus, simply discarding the term $\Delta_1 \ind_{[k,N)}(\xi) \leq 0$ in the first line of \eqref{eq:dht_decoupling}, we deduce that
\begin{align*}
    h'(t)
    &\leq  -\EE \int_0^{2\pi} \int_1^N (1-\cos^2\theta) (Z_t^1 - \tilde{Z}_t^1)^2 \frac{2 \lambda d\theta d\xi}{2\pi(N-1)}
    + \frac{Ck}{N} - \frac{\mu}{2} h(t) \\
    &= -(\lambda + \mu/2) h(t) + \frac{Ck}{N}.
\end{align*}
Thus $h'(t) + (\lambda + \frac{\mu}{2}) h(t) \leq \frac{Ck}{N}$. Since $h(0) = 0$, the desired bound follows from the last inequality by multiplying by $e^{(\lambda + \frac{\mu}{2})t}$ and integrating.
\end{proof}

We now want to obtain an estimate for the decoupling property of the system of Boltzmann processes in terms of $\EE[W_2^2(\bar{\Z}_t, f_t)]$; this is the content of Lemma \ref{lem:decouplingEW2} below. To this end, we will need to recall two results.

For a probability measure $\nu$ on $\RR$ and for any $k\in\NN$, we will let $\varepsilon_k(\nu)$ be given by
\[
    \varepsilon_k(\nu) = \EE[ W_2^2(\bar{\mathbf{X}}, \nu)]
\]
where $\mathbf{X} = (X_1, \dots, X_k)$ is a collection of i.i.d. variables with law $\nu$. The first result, see \cite[Theorem 1]{fournier-guillin2013}, provides rates of convergence for $\varepsilon_k(\nu)$: if $\nu$ has a finite $r^{\text{th}}$ moment for some $r>4$, then there is a constant $C_r$ that depends only on $r$ such that
\begin{equation}
    \label{eq:eps}
    \varepsilon_k(\nu) \leq \frac{C_r \int \vert x \vert^r \nu(dx)}{k^{1/2}}.
\end{equation}
The second result, which is a special case of \cite[Lemma 7]{cortez-fontbona2016}, states that if $\mathbf{X}$ is any exchangeable random vector on $\RR^N$ and $\nu$ is \emph{any} probability measure on $\RR$, then there is a constant $C$ depending only on the second moments of $X^1$ and $\nu$ such that for any $k \leq N$ we have:
\begin{equation}
    \label{eq:W2X}
     \frac{1}{2} \EE[ W_2^2(\bar{\mathbf{X}}, \nu)] \leq W_2^2( \Law(X^1, \dots, X^k), \nu^{\otimes k}) + \varepsilon_k(\nu) + C \frac{k}{N}.
\end{equation}

We are now ready to state and prove:

\begin{lem}
\label{lem:decouplingEW2} 
Assume that $\int_\RR f_0(dv) \vert v \vert^r < \infty$ for some $r>4$. Then there is a constant $C$ depending only on $\lambda$, $\mu$, $T$, $r$, and $\int_\RR f_0(dv) \vert v\vert^r$, such that for all $t\geq 0$ we have
\[
    \EE[ W_2^2( \bar{\mathbf{Z}}_t, f_t) ] \leq \frac{C}{N^{1/3}}.
\]
Moreover, this bound also holds if we replace $\bar{\mathbf{Z}}_t$ by $\bar{\mathbf{Z}}_t^i = \frac{1}{N-1} \sum_{j\neq i} \delta_{Z_t^j}$.
\end{lem}

\begin{proof}
For $k\leq N$, \eqref{eq:W2X} applied to $\nu=f_t$ and $\mathbf{X} = \mathbf{Z}_t$ gives:
\begin{align*}
\frac{1}{2} \EE [W_2^2(\bar{\mathbf{Z}}_t, f_t)]
& \leq W_2^2(\Law(Z_t^1,\ldots,Z_t^k), f_t^{\otimes k})
+ \varepsilon_k(f_t) + C\frac{k}{N} \\
& \leq C\frac{k}{N} + \varepsilon_k(f_t) + C\frac{k}{N},
\end{align*}
where in the last step we used Lemma \ref{lem:decoupling}. The finite initial $r^{\text{th}}$ moment hypothesis, together with Lemma \ref{lem:moments}, implies that 
\[
    \sup_{t\geq 0}\int_\RR \vert v\vert^r f_t(dv) < \infty.
\]
Thus, from \eqref{eq:eps}, we obtain $\varepsilon_k(f_t) \leq C/k^{1/2}$ for all $t\geq 0$ (since $r>4$). Taking $k \sim N^{2/3}$ gives the result. The estimate for $\bar{\mathbf{Z}}_t^i$ is deduced similarly, taking $\mathbf{X} = (Z_t^j)_{j\neq i}$ in \eqref{eq:W2X}.
\end{proof}

We now prove Theorem \ref{thm:UPoC}.

\begin{proof}
Call $h(t) = \EE[(V_t^1-Z_t^1)^2]$. Using Lemma \ref{lem:decouplingEW2} and exchangeability, we obtain
\begin{align*}
\EE [W_2^2(\bar{\mathbf{V}}_t, f_t)]
&\leq 2 \EE [W_2^2(\bar{\mathbf{V}}_t, \bar{\mathbf{Z}}_t)]
     + 2 \EE [W_2^2(\bar{\mathbf{Z}}_t, f_t)] \\
&\leq 2 \EE \left[ \frac{1}{N}\sum_{i=1}^N (V_t^i - Z_t^i)^2 \right]
     + \frac{C}{N^{1/3}} \\
&= 2 h(t) + \frac{C}{N^{1/3}}.
\end{align*}
Thus, it suffices to prove that $h(t) \leq 2 e^{-\frac{\mu}{2} t} h(0) + C N^{-1/3}$, because $h(0) = W_2^2(f_0^N, f_0^{\otimes N})$ thanks to \eqref{eq:W2f0N}.

We thus study the evolution of $h(t)$. We have
\begin{equation*}
h'(t)
= S_t^K + S_t^T.
\end{equation*}
Here $S_t^K$ corresponds to the Kac interactions coming from the $\P_i$ terms in \eqref{eq:sde_vi} and \eqref{eq:sde_zi}, and $S_t^T$ corresponds to the thermostat interactions coming from the $\Q_i$ terms.
For brevity, let us call $V_t^\ii = V_t^{\ii(\xi)}$, $Z_t^\ii = Z_t^{\ii(\xi)}$, and $F_t^1 = F_t^1(\mathbf{Z}_{t^\-},\xi)$. We now study each of $S_t^K$ and $S_t^T$.
For the Kac term $S_t^k$, we recall that the intensity of $\P_1(dt,d\theta,d\xi)$ is $\frac{2 \lambda dt d\theta d\xi \ind_{\{\ii(\xi)\neq 1\}}}{2\pi(N-1)}$. Thus from \eqref{eq:sde_vi} and \eqref{eq:sde_zi}, using It\^{o} calculus, for $S_t^K$ we obtain:
\begin{align}
    \notag
    S_t^K
    &= \EE \int_0^{2\pi} \int_1^N
    \left[ \left(V_t^1 \cos\theta - V_t^{\ii} \sin\theta
    - Z_t^1 \cos\theta + F_t^1 \sin\theta \right)^2 - (V_t^1 - Z_t^1)^2
    \right] \frac{2\lambda d\theta d\xi}{2\pi(N-1)} \\
    \notag
    &= \EE \int_0^{2\pi} \int_1^N
    \left[ (V_t^1 - Z_t^1)^2 (\cos^2\theta -1)  +  (V_t^\ii - F_t^1)^2 \sin^2\theta
    \right] \frac{2\lambda d\theta d\xi}{2\pi(N-1)} \\
    \label{eq:StK_prelim}
    &= 2\lambda \left[
    -\frac{1}{2}h(t) + \frac{1}{2} \EE \int_1^N (V_t^\ii - F_t^1)^2 \frac{d\xi}{N-1}
    \right],
\end{align}
where in the second equality the cross-term vanished since $\int_0^{2\pi} \cos\theta \sin\theta d\theta = 0$. 
We now control the positive term in \eqref{eq:StK_prelim} by subtracting and then adding $Z_t^\ii$ inside the square.
Set $a(t)$ to be $\EE \int_1^N (Z_t^\ii - F_t^1)^2 \frac{d\xi}{N-1}$, thus $a(t) = \EE[W_2^2(\bar{\mathbf{Z}}_t^1, f_t)]$ thanks to \eqref{eq:EW2ftZi}. 
Also note that $\EE \int_1^N (V_t^\ii- Z_t^\ii)^2 \frac{d\xi}{N-1} = \frac{1}{N-1}\sum_{i=2}^N \EE (V_t^j- Z_t^j)^2$ which equals $h(t)$ by exchangeability. Therefore, we have
\begin{align*}
    \EE \int_1^N (V_t^\ii - F_t^1)^2 \frac{d\xi}{N-1}
    &= h(t) + a(t) + 2\EE \int_1^N (V_t^\ii - Z_t^\ii)(Z_t^\ii - F_t^1) \frac{d\xi}{N-1} \\
    &\leq h(t) + a(t) + 2 h(t)^{1/2} a(t)^{1/2},
\end{align*}
where we have used the Cauchy-Schwarz inequality. Plugging this into \eqref{eq:StK_prelim} gives
\[
    S_t^K
    \leq \lambda a(t) + 2\lambda h(t)^{1/2} a(t)^{1/2}.
\]
Next, for the thermostat term $S_t^T$, we recall that the intensity of $\Q_1(dt,d\theta, dw)$ is $\mu dt \frac{d\theta}{2\pi}\gamma(dw)$. Thus, again from \eqref{eq:sde_vi} and \eqref{eq:sde_zi}, we have for $S_t^T$:
\begin{align*}
    S_t^T
    & = \mu \EE\int_\RR \int_0^{2\pi} \left[ (V_t^1 \cos\theta - w \sin\theta - Z_t^1 \cos\theta + w\sin\theta)^2 - (V_t^1-Z_t^1)^2 \right] \frac{d\theta}{2\pi} \gamma(dw) \\
    &= -\frac{\mu}{2} h(t).
\end{align*}
Joining the bounds for $S_t^K$ and $S_t^T$, we see that
\begin{equation}
    \label{eq:dht_UPoC_final}
    h'(t) \leq -\frac{\mu}{2}h(t) + \lambda a(t) + 2\lambda  h(t)^{1/2} a(t)^{1/2}.
\end{equation}
Lemma \ref{lem:decouplingEW2} showed that $a(t) \leq C/N^{1/3}$. Thus, the Theorem follows from \eqref{eq:dht_UPoC_final} by a Gronwall-type inequality (see for example \cite[Lemma 4.1.8]{ambrosio-gigli-savare2008}).
\end{proof}

\section{Conclusion}
\label{sec:conclusion}



In this work we showed that the thermostated Kac $N$-particle system propagates chaos uniformly in time, at a polynomial rate of $N^{-1/3}$ for the $2$-Wasserstein metric, improving the propagation of chaos result in \cite{bonetto-loss-vaidyanathan2014}. This illustrates that the coupling method in \cite{cortez-fontbona2016} can be adapted to include thermostats.  We also use coupling arguments to deduce equilibration estimates for both the particle system and the kinetic equation.

We plan on developing this coupling method further to a Kac-type model where, in addition to the particle collisions \eqref{eq:kac} and the thermostat interactions \eqref{eq:thermostat}, the system has an \emph{energy restoring mechanism} that pushes the total energy of the system to its initial value after each interaction with the thermostat. This is the subject of future research.

\bigskip

\paragraph{Acknowledgements.}
We would like to thank Federico Bonetto and Joaquin Fontbona for fruitful discussions during our work.

\bibliographystyle{plain}
\bibliography{references.bib}{}

\begin{thebibliography}{10}

\bibitem{ambrosio-gigli-savare2008}
Luigi Ambrosio, Nicola Gigli, and Giuseppe Savar{\'e}.
\newblock {\em Gradient flows in metric spaces and in the space of probability
  measures}.
\newblock Lectures in Mathematics ETH Z\"urich. Birkh\"auser Verlag, Basel,
  second edition, 2008.

\bibitem{bonetto-loss-tossounian-vaidyanathan2017}
F.~Bonetto, M.~Loss, H.~Tossounian, and R.~Vaidyanathan.
\newblock Uniform approximation of a {M}axwellian thermostat by finite
  reservoirs.
\newblock {\em Comm. Math. Phys.}, 351(1):311--339, 2017.

\bibitem{bonetto-loss-vaidyanathan2014}
Federico Bonetto, Michael Loss, and Ranjini Vaidyanathan.
\newblock The {K}ac model coupled to a thermostat.
\newblock {\em J. Stat. Phys.}, 156(4):647--667, 2014.

\bibitem{carlen-carvalho-loss2000}
Eric Carlen, M.~C. Carvalho, and Michael Loss.
\newblock Many-body aspects of approach to equilibrium.
\newblock In {\em Journ\'ees ``\'{E}quations aux {D}\'eriv\'ees {P}artielles''
  ({L}a {C}hapelle sur {E}rdre, 2000)}, pages Exp.\ No.\ XI, 12. Univ. Nantes,
  Nantes, 2000.

\bibitem{carlen-carvalho-leroux-loss-villani2010}
Eric~A. Carlen, Maria~C. Carvalho, Jonathan Le~Roux, Michael Loss, and
  C{\'e}dric Villani.
\newblock Entropy and chaos in the {K}ac model.
\newblock {\em Kinet. Relat. Models}, 3(1):85--122, 2010.

\bibitem{cortez2016}
Roberto Cortez.
\newblock Uniform propagation of chaos for {K}ac's 1{D} particle system.
\newblock {\em J. Stat. Phys.}, 165(6):1102--1113, 2016.

\bibitem{cortez-fontbona2016}
Roberto Cortez and Joaquin Fontbona.
\newblock Quantitative propagation of chaos for generalized {K}ac particle
  systems.
\newblock {\em Ann. Appl. Probab.}, 26(2):892--916, 2016.

\bibitem{cortez-fontbona2018}
Roberto Cortez and Joaquin Fontbona.
\newblock Quantitative uniform propagation of chaos for maxwell molecules.
\newblock {\em Communications in Mathematical Physics}, 357(3):913--941, Feb
  2018.

\bibitem{einav2011}
Amit Einav.
\newblock On {V}illani's conjecture concerning entropy production for the {K}ac
  master equation.
\newblock {\em Kinet. Relat. Models}, 4(2):479--497, 2011.

\bibitem{fournier-guillin2013}
Nicolas Fournier and Arnaud Guillin.
\newblock On the rate of convergence in {W}asserstein distance of the empirical
  measure.
\newblock {\em Probability Theory and Related Fields}, 162(3-4):707--738, 2015.

\bibitem{hauray2016}
Maxime Hauray.
\newblock Uniform {C}ontractivity in {W}asserstein {M}etric for the {O}riginal
  1{D} {K}ac's {M}odel.
\newblock {\em J. Stat. Phys.}, 162(6):1566--1570, 2016.

\bibitem{janvresse2001}
Elise Janvresse.
\newblock Spectral gap for {K}ac's model of {B}oltzmann equation.
\newblock {\em Ann. Probab.}, 29(1):288--304, 2001.

\bibitem{kac1956}
M.~Kac.
\newblock Foundations of kinetic theory.
\newblock In {\em Proceedings of the {T}hird {B}erkeley {S}ymposium on
  {M}athematical {S}tatistics and {P}robability, 1954--1955, vol. {III}}, pages
  171--197, Berkeley and Los Angeles, 1956. University of California Press.

\bibitem{mischler-mouhot2013}
St{\'e}phane Mischler and Cl{\'e}ment Mouhot.
\newblock Kac's program in kinetic theory.
\newblock {\em Invent. Math.}, 193(1):1--147, 2013.

\bibitem{sznitman1989}
Alain-Sol Sznitman.
\newblock Topics in propagation of chaos.
\newblock In {\em {\'E}cole d'{\'E}t{\'e} de {P}robabilit{\'e}s de
  {S}aint-{F}lour {XIX}---1989}, volume 1464 of {\em Lecture Notes in Math.},
  pages 165--251. Springer, Berlin, 1991.

\bibitem{tanaka1978}
Hiroshi Tanaka.
\newblock On the uniqueness of {M}arkov process associated with the {B}oltzmann
  equation of {M}axwellian molecules.
\newblock In {\em Proceedings of the {I}nternational {S}ymposium on
  {S}tochastic {D}ifferential {E}quations ({R}es. {I}nst. {M}ath. {S}ci.,
  {K}yoto {U}niv., {K}yoto, 1976)}, pages 409--425. Wiley, New
  York-Chichester-Brisbane, 1978.

\bibitem{tanaka.s1968}
Shigeru Tanaka.
\newblock An extension of {W}ild's sum for solving certain non-linear equation
  of measures.
\newblock {\em Proc. Japan Acad.}, 44:884--889, 1968.

\bibitem{tossounian-vaidyanathan2015}
Hagop Tossounian and Ranjini Vaidyanathan.
\newblock Partially thermostated {K}ac model.
\newblock {\em J. Math. Phys.}, 56(8):083301, 16, 2015.

\bibitem{villani2009}
C{\'e}dric Villani.
\newblock {\em Optimal transport, old and new}, volume 338 of {\em Grundlehren
  der Mathematischen Wissenschaften [Fundamental Principles of Mathematical
  Sciences]}.
\newblock Springer-Verlag, Berlin, 2009.

\end{thebibliography}

\end{document}